\documentclass[12pt]{amsart}
\usepackage{pdfsync}
\usepackage{graphicx,color}
\usepackage{amssymb}
\usepackage{amsmath}
\usepackage{amsthm}
\usepackage{enumerate}
\usepackage{amsthm}
\usepackage{esvect}

\usepackage{ucs}
\usepackage{cite}
\usepackage{amsxtra}
\usepackage{epstopdf}
\usepackage{bbold}
\usepackage[utf8x]{inputenc}
\usepackage{verbatim} 

\usepackage[toc,page]{appendix}
\usepackage[left=3cm,top=2.5cm,right=3cm,bottom=2.5cm]{geometry}

\newcommand{\dist}{\text{dist}}

\newtheorem*{theorem*}{Theorem} 

\newtheorem{theorem}{Theorem}
\newtheorem{definition}[theorem]{Definition}
\newtheorem{proposition}[theorem]{Proposition}
\newtheorem{corollary}[theorem]{Corollary}
\newtheorem{observation}[theorem]{Observation}
\newtheorem{lemma}[theorem]{Lemma}
\newtheorem{example}[theorem]{Example}

\newtheorem{remark}[theorem]{Remark}

\newcommand{\R}{\mathbb{R}}
\def\Z{{\mathbb Z}}
\def\N{{\mathbb N}}

\makeatletter
\def\moverlay{\mathpalette\mov@rlay}
\def\mov@rlay#1#2{\leavevmode\vtop{%
   \baselineskip\z@skip \lineskiplimit-\maxdimen
   \ialign{\hfil$\m@th#1##$\hfil\cr#2\crcr}}}
\newcommand{\charfusion}[3][\mathord]{
    #1{\ifx#1\mathop\vphantom{#2}\fi
        \mathpalette\mov@rlay{#2\cr#3}
      }
    \ifx#1\mathop\expandafter\displaylimits\fi}
\makeatother

\definecolor{Azul}{rgb}{0.0, 0.0, 1.0}
\definecolor{Rojo}{rgb}{1.0, 0.03, 0.0}
\definecolor{Purpura}{rgb}{0.5,0,0.35}
\definecolor{BurntOrange}{cmyk}{0,0.51,1,0}
\definecolor{PineGreen}{cmyk}{0.92,0,0.59,0.250}
\definecolor{Violeta}{rgb}{0.39,0.17,0.63}
\definecolor{Violeta2}{rgb}{0.6, 0.4, 0.8}
\definecolor{Fucsia}{rgb}{1.0, 0.01, 0.24}
\definecolor{Rosa}{rgb}{0.63,0.17,0.39}
\definecolor{VerdeManzana}{rgb}{0.39,0.63,0.17}
\definecolor{Celeste}{rgb}{0.7,0.7,1}

\begin{document}
\thispagestyle{empty}

\title{Patterns in thick compact sets}

\author{Alexia Yavicoli}

\address{School of Mathematics and Statistics\\
University of St. Andrews}

\email{ay41@st-andrews.ac.uk or alexia.yavicoli@gmail.com}

\thanks{This work was partially supported by CONICET and the Swiss National Science Foundation, grant n$^{\circ}$ P2SKP2\_184047}


\begin{abstract}
We introduce a connection between Newhouse thickness and patterns through a variant of Schmidt's game introduced by Broderick, Fishman and Simmons. This yields an explicit, robust and checkable condition that ensures that a Cantor set in the real line contains long arithmetic progressions and, more generally, homothetic copies of large finite sets.
\end{abstract}

\maketitle

\tableofcontents

\section{Introduction and main results}

\subsection{Introduction}

It is a well known consequence of the Lebesgue density theorem that any set $E \subseteq \R^N$ of positive Lebesgue measure contains homothetic copies of every finite set. A very active area of research concerns finding conditions of structure or size on sets of zero measure that imply the existence of certain patterns.  The most basic example of a pattern of interest is an arithmetic progression; note that an $m$-term arithmetic progression is a homothetic copy of the set $\{0,1,\ldots,m-1\}$.

Perhaps the most natural notion of size to consider within sets of zero Lebesgue measure is Hausdorff dimension. However, Keleti \cite{Kel2} constructed a compact subset of the real line of Hausdorff dimension $1$ that does not contain any arithmetic progressions of length $3$. This result was extended and refined by Keleti \cite{Kel1}, M\'{a}th\'{e} \cite{Mat} and the author \cite{AY17}. In the latter paper we showed that given a dimension function $h(x)$ smaller than $x$ in the natural order, there is a set of positive Hausdorff $h$-measure that does not contain any $3$-term progression (or more generally, that avoids countably many linear patterns given in advance).

In the opposite direction, it follows from results of Davies, Marstrand and Taylor \cite{DMT} and of Molter and the author \cite{UMAY} that there are perfect sets of Hausdorff dimension $0$ in the real line containing a homothetic copy of every finite set, and even every polynomial pattern.

These results indicate that Hausdorff measure and dimension cannot, in itself, detect the presence or absence of patterns in sets of measure zero, even in the most basic case of arithmetic progressions. Before continuing, we remark that large Hausdorff dimension does imply the presence of other kinds of patterns such as angles, triangles or certain polynomial configurations - see for example \cite{HKKMMMS,IosLiu,Krause19}. Moreover, Hausdorff dimension is crucial in the closely related problem of finding an abundance of configurations rather than one configuration in particular. The most classical example is the Falconer distance set problem, that asks what lower bound on the Hausdorff dimension of a set in $\mathbb{R}^d$, $d\ge 2$ implies that the sets of distances determined by the set has positive Lebesgue measure. See, for example, \cite{IosTay,GILP, GILP16,GIM18,GIOW} and references there for some examples of this fruitful line of research, covering distance sets and some natural generalizations such as ``simplex sets''.

In this paper we are concerned with sets in the real line containing long arithmetic progressions (and, more generally, homothetic copies of large finite sets). As we saw, for this problem Hausdorff dimension alone does not provide any information. {\L}aba and Pramanik \cite{LP09} showed that if in addition to having large Hausdorff dimension (in a quantitative sense), a subset of the line supports a probability measure obeying appropriate Fourier decay conditions, then it contains arithmetic progressions of length $3$. The hypotheses were relaxed, and the family of patterns covered greatly enlarged, in the subsequent papers \cite{HLP16, CLP16, FGP19}. Their techniques are based on harmonic analysis, and these methods do not usually work for longer arithmetic progressions. Moreover, the hypotheses are also difficult to check, and in fact are not known to hold or outright fail for natural classes of fractals such as central self-similar Cantor sets.

The goal of this paper is to show that Newhouse thickness, which is a different, well-known notion of size introduced by Newhouse \cite{Newhouse} in 1970, does allow to detect the presence of long arithmetic progressions (and also homothetic and more general copies of finite sets) inside fractal sets in the real line. One key advantage of Newhouse thickness is that it is easy to compute or estimate for many classical fractal sets such as self-similar sets of sets defined in terms of continued fraction coefficients. (A drawback is that no satisfactory notion of thickness exists in higher dimensions.)

\subsection{Statement of main result}

We review the definition of thickness. To begin, recall that every compact set $C$ on the real line can be constructed by starting with a closed interval $I$ (its convex hull), and successively removing disjoint open complementary intervals (they are the path-connected components of the complement of $C$). Clearly there are finitely or countably  many disjoint open complementary intervals $(G_n)_n$, which we may assume are ordered so that their length $|G_n|$ is decreasing (if some intervals have the same length, we order them arbitrarily). We emphasize that the two unbounded connected components of $\mathbb{R}\setminus C$ are not considered as we start with the convex hull $I$. Note that $G_{n+1}$ is a subset of some connected component $I_{n+1}$ (a closed interval) of $I\setminus (G_1\cup\cdots\cup G_n)$. We say that $G_{n+1}$ is removed from $I_{n+1}$.

\begin{definition}
Let $C \subset \R$ a compact set with convex hull $I$, and let $(G_n)$ be the open intervals making up $I\setminus C$, ordered in decreasing length. Each $G_n$ is removed from a closed interval $I_n$, leaving behind two closed intervals $L_n$ and $R_n$; the left and right pieces of $I_n \setminus G_n$.

We define the thickness of $C$ as
\[
\tau (C):= \inf_{n \in \N} \frac{\min \{ |L_n| , |R_n| \}}{|G_n|}.
\]
Note that the sequence of complementary intervals $(G_n)_n$ may be finite, and in this case the infimum is taken over the finite set of indices.

We define the thickness of a singleton as $0$, and the thickness of a (non-degenerate) interval as $+\infty$.
\end{definition}
It can be checked that if there are some intervals of equal length, then the way in which we order them does not affect the value of $\tau(C)$. See \cite{HKY93,Astels} for more information on Newhouse thickness and alternative definitions.

\begin{example}\label{excantor}
Let $M_{\varepsilon}$ be the middle-$\varepsilon$ Cantor set obtained by starting with the interval $[0,1]$ and repeatedly deleting from each interval appearing in the construction the middle open interval of relative length $\varepsilon$. Then $\tau (M_{\varepsilon})= \frac{1-\varepsilon}{2 \varepsilon}$.
\end{example}

Intuitively, thickness is a measure of how large the compact set is relative to the intervals in its complement. If a set has large thickness, then it also has large Hausdorff dimension. In fact (see \cite[page 77]{PalisTakens}),
\[
\text{dim}_{\rm H} (E) \geq \frac{\log(2)}{\log(2+\frac{1}{\tau (E)})}.
\]
On the other hand, a set with an isolated point has thickness zero. Similarly, one can construct a topological Cantor set of positive Lebesgue measure (in particular of full Hausdorff dimension) with thickness $0$.

The constructions showing that a set of large Hausdorff dimension/measure can avoid arithmetic progressions \cite{Kel2, Kel1, Mat,AY17} rely on ``killing'' the progressions at stages of the construction where the set looks very thin. So they exploit the fact that a set of large Hausdorff dimension may have some scales at which the set is very sparse. Unlike Hausdorff dimension, however, large thickness implies that the set is robustly ``dense'' at \emph{all} scales and positions, because it is defined as an infimum over all removed intervals $G_n$.

Newhouse defined thickness motivated by problems in dynamical systems where it is important to know that two Cantor sets intersect robustly. He proved the following famous Gap Lemma:
\begin{theorem}[Newhouse's Gap Lemma]
Given two Cantor sets $C_1, C_2 \subset \R$, such that neither set lies in a gap of the other and \textbf{$\tau(C_1) \tau(C_2) > 1$},
\[
C_1 \cap C_2 \neq \emptyset.
\]
\end{theorem}

The study of patterns is closely linked with that of multiple intersections. For example, $E$ contains an arithmetic progression of length $m$ and gap $\Delta$ if  $\bigcap_{0\leq k \leq m-1} (E-k\Delta) \neq \emptyset$. As another example, $E$ contains a translated copy of the finite set $\{x_1, \cdots, x_k\}$ if and only if $\bigcap_{1\leq i \leq k}E-x_i \neq \emptyset$.

Unfortunately, the Gap Lemma does not generalize in any simple way to intersections of three or more sets. (A rather cumbersome way to do this would be via the results from \cite{HKY93} on thickness of intersections, but in the best case this would lead to results that are quantitatively much poorer than those in this paper.) So, if we want to guarantee the presence of patterns inside sets, we need a different approach. Another notion of ``largeness'' for sets of zero measure is given by the theory of Schmidt-type games. For the original game, Schmidt-winning sets are countably stable under intersection, despite being (possibly) of zero Lebesgue measure - a property that clearly fails for sets of large or even full Hausdorff dimension, and even for sets of positive (but not full) Lebesgue measure. However, for the original Schmidt's game, winning sets have full Hausdorff dimension and this excludes most fractals of interest. More recent, quantitative variants of Schmidt's game, however, allow for sets of Hausdorff dimension less than full; in particular, this is the case for the potential game introduced in \cite{BFS}. The key property that makes winning sets useful for studying patterns is that, almost by definition, the countable intersection of winning sets is again winning (with different parameters): see Proposition \ref{Countable intersection property}. The main contribution of this paper is to relate thickness and patterns via Schmidt-type winning sets:

\begin{theorem}\label{finiteexplicit}
Let $C\subset\mathbb{R}$ be a compact set. Then $C$ contains a homothetic copy of every set with at most
\[
N(\tau):=\left\lfloor\frac{\log (4)}{4 e (720)^2}\frac{\tau}{\log (\tau)}\right\rfloor
\]
elements. Moreover, for each such set $A$, the compact set $C$ contains $\lambda A+x$ for some $\lambda>0$ and a set of $x$ of positive Hausdorff dimension.
\end{theorem}
We make some remarks on this statement.
\begin{enumerate}
 \item The theorem has no content when $N(\tau)<3$, that is, for sets whose thickness is below some universal constant $\tau_3$ satisfying $N(\tau_3)=3$. On the other hand, it shows that for every $m$, if the thickness of $E$ is larger than some (finite) constant $\tau_m$, then $E$ contains progressions of length $m$ as well as homothetic copies of \emph{every} set with $m$ elements. This includes many classical fractal sets.
  \item The value of the constant in $N(\tau)$ is very poor and we have made no attempt at optimizing it, but nevertheless it is an explicit constant (while in many papers on the existence of patterns inside ``large'' sets, the threshold for largeness is not given explicitly, see for example \cite{LP09, HLP16, CLP16, FGP19, BFS, IosLiu}).
  \item When the diameter of the finite set is smaller than $\text{diam}(C)/8$, the proof shows that we can take $\lambda=1$, that is, we find translated copies instead of homothetic ones.
  \item Recall the definition of the central Cantor set $M_{\varepsilon}$ from Example \ref{excantor}. It was shown in  \cite[Theorem 2.1]{BFS}) that if $L_{\varepsilon}$ is the length of the longest arithmetic progression contained in $M_{\varepsilon}$ then, for $\varepsilon$ small,
      \[
      \frac{\frac{1}{\varepsilon}}{\log(\frac{1}{\varepsilon})} \lesssim L_{\varepsilon}\lesssim \frac{1}{\varepsilon}.
      \]
      (Here we used the standard notation $A\lesssim B$ if there exists a positive constant $K$ such that $A \leq KB$, and $A \sim B$ if $A\lesssim B$ and $B\lesssim A$ simultaneously.) Note that Theorem \ref{finiteexplicit} recovers the upper bound. This example shows that, in general, Theorem \ref{finiteexplicit} gives fairly efficient bounds if $\tau$ is very large.
\end{enumerate}

We can extend the previous theorem to more general bi-Lipschitz patterns.

\begin{definition}
We say $f: \R \to \R$ is a bi-Lipschitz function with constants $(c_1, c_2)$, if
\[c_1 |x-y| \leq |f(x)-f(y)|\leq c_2 |x-y| \text{ for all } x,y \in \R.\]
\end{definition}

\begin{definition}\label{defpattern}
Given $E \subset \R$, we say that $E$ contains the pattern $(f_i)_{i \in \Lambda}$ if
$$\text{there exists } t \in \R \text{ such that } f_i(t) \in E \   \forall i \in \Lambda,$$ or equivalently, if
 \begin{equation*}\bigcap_{i \in \Lambda} f^{-1}_i (E) \neq \emptyset.
 \end{equation*}

 In the particular case where the functions $f_i$ are bi-Lipschitz functions, we say that the pattern is a bi-Lipschitz pattern.
\end{definition}

We denote the convex hull of a set $C$ by $\text{conv}(C)$.

\begin{theorem}\label{BilipExplicit}Given constants $A \geq 1$, $D>0$, $m>0$, the following holds. Let $\mathcal{F}$ be a family of bi-Lipschitz functions with constants $(c_1(f), c_2(f))$ such that $\frac{c_2(f)}{c_1(f)}\leq A$, $(c_1(f))^{-1}\leq D$ and such that there exists a closed interval $I$ of length $m>0$ contained in $\bigcap_{f \in \mathcal{F} } f^{-1}([0,1])$.

Let
\[
N(\tau):=\left\lfloor \frac{\tau \beta}{720^2 e  A^{1-\frac{1}{\log(\tau \beta)} }} \left( 1 − \beta^{\frac{1}{\log(\tau \beta)}} \right) \right\rfloor,
\]
where we define $\beta:= \min\{ \frac{m}{D}; \frac{1}{4A}\}$.

Let $C\subset\mathbb{R}$ be a compact set with $\text{conv}(C)=[0,1]$ and let $S=(-\infty,0)\cup C\cup (0,\infty)$. Then for any family $\mathcal{F}_0\subset\mathcal{F}$ with $\le N(\tau(C))$ elements,
\[
\dim_{\rm H}\left(I\cap \bigcap_{f\in\mathcal{F}_0} f^{-1}(S)\right) > 0.
\]
In particular, since $S\cap f(I)\subset S\cap [0,1]\subset C$ for all $f\in\mathcal{F}$, the set $C$ contains every pattern in the family $\mathcal{F}$ with at most $N(\tau)$ elements, and in fact each such patterns is obtained for a set of $t$ (from Definition \ref{defpattern}) of positive Hausdorff dimension.

When $\tau$ is large enough, we also have that $N(\tau) \sim \frac{\tau}{\log (\tau)}$.
\end{theorem}
The assumption that the convex hull is $[0,1]$ can be removed, see Corollary \ref{bilipGral} below.

\section{Examples, remarks and generalizations}

In this section we present some applications and generalizations of the main results. We start with some straightforward but useful observations.

\begin{remark}
Thickness is not monotone with respect to inclusion. However, in order to find patterns inside a set, it is clearly sufficient to find them inside a subset. In other words, Theorems \ref{finiteexplicit} and \ref{BilipExplicit} hold more generally for
\[
\tilde{\tau}(E):=\sup_{A \subseteq E \atop A\text{ compact set}}\tau(A)
\]
in place of $\tau(E)$.
\end{remark}

\begin{remark} \label{remloc}
We define the local thickness of $E$ as
\[
\tau_{\text{loc}}(E):=\sup_{x\in E}\limsup_{r \to 0+}\tau(E\cap B(x,r)).
\]
Clearly, $\tau_{\text{loc}}(E)\ge \tilde{\tau}(E)$ and hence Theorems \ref{finiteexplicit} and \ref{BilipExplicit} hold for $\tau_{\text{loc}}$ in place of $\tau$. While in general $\tau$ is not invariant under $C^1$ diffeomorphisms, it is easy to see that $\tau_{\text{loc}}(f(E))=\tau_{\text{loc}}(E)$ for every $C^1$ diffeomorphism $f$ of $\R$. Then, we get the presence of homothetic copies of finite sets (of size independent of $f$) in the set $f(E)$, where $f$ is a $C^1$-diffeomorphism.
\end{remark}

Let us now discuss some concrete examples. We recall an equivalent and useful definition of thickness :
We say $S_1$ is a chunk of $S_2$, and write $S_1 \varpropto S_2$, if $S_1$ is the intersection of a closed interval with $S_2$, is a proper subset of $S_2$, and the distance between $S_1$ and $S_2\setminus S_1$ is positive. Note that a closed set $C$ has a chunk if and only if it is not connected. Thickness can also be computed as
\[
\tau(C)=\inf_{S \varpropto C} \frac{|S|}{\dist(S, C\setminus S)}.
\]
See e.g. \cite[Proposition 2]{HKY93}. Using this variant, it is easy to estimate the thickness of self-similar sets. Recall that a compact set $C\subset \R$ is \emph{self-similar} if there exist finitely many similarities $f_1,\cdots, f_k$ such that
\[
C=\bigcup_{1\le i\le k} f_i(C).
\]

\begin{example}
Let $C\subset\R$ be a self-similar set with convex hull $I$. Assume that the intervals $f_i(I)$ are disjoint and without generality assume that they are ordered from left to right. Let $\lambda_1, \cdots, \lambda_k$ the length of the children of the first level in the construction, that is, of the intervals $f_i(I)$, and let $h_{i, i+1}$ denote the gap between the intervals $f_i(I)$ and $f_{i+1}(I)$. Then, by self-similarity and the equivalent definition of thickness, we have
\[
\tau(C)\geq \min \left\{\frac{\lambda_1}{h_{1,2}},\frac{\lambda_2}{\min\{h_{1,2}, h_{2,3}\}}, \cdots,  \frac{\lambda_{k-1}}{\min\{h_{k-2,k-1}, h_{k-1,k}\}} , \frac{\lambda_k}{h_{k-1,k}} \right\}.
\]
\end{example}

Note that this example substantially generalizes the middle-Cantor sets $M_{\varepsilon}$ from Example \ref{excantor}. Together with Theorem \ref{finiteexplicit}, it shows that if the gaps between intervals of the construction are small compared to the length of the intervals, then self-similar sets contain homothetic copies of all finite sets of a given size. We also observe that for self-similar sets, local and global thickness clearly agree, and hence Theorem \ref{finiteexplicit} and Remark \ref{remloc} show that diffeomorphic images of self-similar sets of large thickness also contain long arithmetic progressions (with the length independent of the diffeomorphism).

There is also much interest in fractals arising from non-linear dynamics, and among them one of the most classical classes of examples are fractals defined in terms of the continued fraction expansion. Given $B \subseteq \N$ we define the Cantor set $E(B)$ which consists of all numbers in $(0,1)$ with continued fraction coefficients contained in $B$. In \cite[Lemmas 4.2 and 4.3]{Astels} Astels obtained explicit lower bounds for $\tau(E(B))$ (they are rather lengthy to state in detail). In combination with Theorem \ref{finiteexplicit}, these bounds show that in many cases $E(B)$ contains long arithmetic progressions and other finite patterns. In the particular case $B:=\{1, \cdots, n\}$, this recovers and extends a bound from \cite[Theorem 1.4]{BFS}, but we emphasize that $B$ can be a more general, even infinite, set.

An important problem in additive combinatorics is the study of configurations in sumsets. Astels \cite{Astels} proved the following lower bound on the thickness of sumsets:
\begin{theorem}[\cite{Astels}, Theorem 2.4 (3)]
Let $E_1, \cdots, E_k \subset \R$ be Cantor sets, and let $s$ be defined by
\[
s = \frac{\tau(E_1)}{\tau(E_1)+1}+ \cdots + \frac{\tau(E_k)}{\tau(E_k)+1}.
\]
If $s\ge 1$ then $E_1+\cdots +E_k$ contains an interval; otherwise, $E_1+\cdots+E_k$ contains a Cantor set of thickness $\ge s/(1-s)$.
\end{theorem}
So, by using any of our theorems, we can guarantee the presence of arithmetic progressions or homothetic copies of finite sets in sums of Cantor sets. This is non-trivial when the value $s$ from the theorem is close to but smaller than $1$. We note that there is a large literature on the dimension, measure and interior of sumsets of Cantor sets in the line (see, for example, \cite{CHM, PeresSolomyak, PeresShmerkin}), however many (but not all) of these results hold generically and give no information for explicit Cantor sets, or they deal with Hausdorff dimension only, which as we have seen on its own gives no information on the presence of arithmetic progressions.

In Theorem \ref{BilipExplicit}, we assumed that the convex hull of the compact set is $[0,1]$. In the following corollary, we get rid of this condition:
\begin{corollary}\label{bilipGral}
Given constants $A \geq 1$, $D>0$, $m>0$, $\lambda>0$, $t\in \R$ the following holds. Let $\mathcal{F}$ be a family of bi-Lipschitz functions with constants $(c_1(f), c_2(f))$ such that $\frac{c_2(f)}{c_1(f)}\leq A$, $\lambda(c_1(f))^{-1}\leq D$ and such that there exists a closed interval $I$ of length $m>0$ contained in $\bigcap_{f \in \mathcal{F} } f^{-1}(\lambda[0,1]+t)$.

Let
\[
N(\tau):=\left\lfloor \frac{\tau \beta}{720^2 e  A^{1-\frac{1}{\log(\tau \beta)} }} \left( 1 − \beta^{\frac{1}{\log(\tau \beta)}} \right) \right\rfloor,
\]
where $\beta:= \min\{ \frac{m}{D}; \frac{1}{4A}\}$.

Let $C\subset\mathbb{R}$ be a compact set with $\text{conv}(C)=[t,t+\lambda]$. Then
\[
\dim_{\rm H}\left(I\cap \bigcap_{f\in\mathcal{F}_0} f^{-1}\big((-\infty, t)\cup C\cup (t+\lambda,\infty)\big)\right) > 0.
\]
for all subfamilies $\mathcal{F}_0$ of $\mathcal{F}$ with at most $N(\tau(C))$ elements.
\end{corollary}

\begin{proof}
We define $T(x):=\frac{1}{\lambda}(x-t)$ and $E:=T(C)$. We have that $\tau (E)=\tau (C)$ and $\text{conv}(E)=[0,1]$.


By applying Theorem \ref{BilipExplicit} to $\widetilde{\mathcal{F}}:=\{T \circ f:f\in\mathcal{F}\}_i$ and $E$ (that satisfy all the hypotheses), we get that
\[
\dim_{\rm H}\left(I\cap  \bigcap_{i=1}^k (T\circ f_i)^{-1}\big((-\infty,0)\cup E \cup (0,\infty)\big)  \right)>0
\]
whenever $f_i\in\mathcal{F}$ and $k\le N(\tau(E))=N(\tau(C))$. But
\[
(T\circ f_i)^{-1}\big((-\infty,0)\cup E \cup (0,\infty)\big) = f_i^{-1} \left((-\infty, t)\cup C\cup (t+\lambda,\infty)\right),
\]
so this completes the proof.
\end{proof}

Besides linear patterns such as arithmetic progressions, polynomial patterns have attracted a lot of attention. Even though polynomial functions are not bi-Lipschitz in general, by considering intervals on which they are injective, we can still apply Corollary \ref{bilipGral} in order to ensure polynomial patterns in thick compact sets. Let $P_1, \cdots, P_k$ be a family of polynomials (or more general smooth functions) which are all injective on some interval $I$ of positive Lebesgue measure, and suppose
\[
0<c_1\leq P_i'|_I \leq c_2<+\infty
\]
and that $P_i^{-1}([b, b+1]) \supseteq I$ for some $b$. We can then extend $P_i|_I$ to a $(c_1,c_2)$-bi-Lipschitz function $\overline{P}_i$ on all of the real line. Let $C$ be a compact set whose convex hull is $[b,b+1]$. By Corollary \ref{bilipGral}, if $\tau(C)$ is large enough in terms of $k$, then there exists
\[
t \in I\cap \bigcap_{1\leq i \leq k} \overline{P}_i^{-1}\big((-\infty,b)\cup C\cup (b+1,\infty)\big).
\]
Since $\overline{P}_i(I)=P_i(I)\subset [b,b+1]$, we get that $P_i(t)\in C$ for all $i$.

As a concrete example, we consider a family of quadratic functions
\[
P_i(x):=(x-x_i)^2+y_i,
\]
where $y_i \geq 0$, and
\[\max_i \sqrt{b-y_i}+x_i< \min_j \sqrt{b+1-y_j}+x_j
\]
for some $b>\max\{x_1, \cdots, x_n, y_1, \cdots, y_n \}$. This is clearly possible if the $x_i,y_i$ are small enough in terms of $b$. Then every sufficiently thick compact set $C$ with $\text{conv}(C)=[b,b+1]$ contains the given quadratic pattern $\{P_1, \cdots, P_n\}$. We can easily check that the required conditions are satisfied by taking
\[
I:=[\max_i \sqrt{b-y_i}+x_i, \min_j \sqrt{b+1-y_j}+x_j]
\]
and $0<c_1:=2\sqrt{b-\max \{y_1, \cdots, y_n\}}\leq c_2:=2\sqrt{b+1}<+\infty$.


\section{The potential game}

We introduce a particular case of a game of Schmidt type defined in \cite{BFS}.

\begin{definition}
Given $\alpha, \beta, \rho >0$ and $c \geq 0$, Alice and Bob play the $(\alpha, \beta, c, \rho)$-game in the real line under the following rules:
\begin{itemize}
\item For each $m \in \N_{0}$ Bob plays first, and then Alice plays.
\item On the $m$-th turn, Bob plays a closed ball $B_m:=B[x_m , \rho_m ]$, satisfying $\rho_0 \geq \rho$, and $\rho_{m}\geq \beta \rho_{m-1}$ and $B_m \subseteq B_{m-1}$ for every $m \in \N$.
\item On the $m$-th turn Alice responds by choosing and erasing a finite or countably infinite collection $\mathcal{A}_m = A(\rho_{i,m} , h_{i,m})$ of balls of radius $\rho_{i,m}>0$ and center $h_{i,m}$. Alice's collection must satisfy $\sum_{i} \rho_{i,m}^c \leq (\alpha \rho_m )^c$ if $c>0$, or $\rho_{1,m} \leq \alpha \rho_m$ if $c=0$ (in the case $c=0$ Alice can erase just one set).
\item $\lim_{m \to \infty} \rho_m =0$ (Note that this is a non-local rule for Bob. One can define this game without this rule, adding that Alice wins if $\lim_{m \to \infty} \rho_m \neq 0$. But, to make the definitions simpler we added this condition as a rule for Bob.)
\end{itemize}
\end{definition}

In any case, Alice is allowed not to erase any set, or equivalently to pass her turn.

There exists a single point $x_{\infty} = \bigcap_{m \in \N_0} B_m$ called the outcome of the game. We say a set $S \subset \R$ is a winning set if Alice has a strategy guaranteeing that if $x_{\infty} \notin \bigcup_{m \in \N_0} \bigcup_i A(\rho_{i,m} , h_{i,m})$, then $x_{\infty} \in S$.

Note that the conditions $B_0 \supseteq B_1 \supseteq \cdots$ and $\lim_{m \to \infty} \rho_m =0$ imply $\beta < 1$.

Intuitively, the balls chosen by Bob determine where to ``zoom in'' in the set, and Alice erases as much as possible from the complement of $S$. If $S$ was a self-similar Cantor set such as the sets $M_{\varepsilon}$ from Example \ref{excantor}, Bob would choose gaps in the construction of the Cantor set, and Alice would try to erase as much as possible of these gaps. The set $S$ is winning if for any sequence of ``zooming-ins'' chosen by Bob, Alice has a strategy of erasures that allows her to either erase the point $x_\infty$ or ensure it ends up lying in $S$.

The original game in \cite{BFS} allowed Alice to answer with $\rho_{i,n}$-neighborhoods of sets in a given family, and the ambient space was $\R^d$. In our case the family is the family of every singleton in $\R$.

The following properties are easy to see (see \cite{BFS}):
\begin{proposition}[Countable intersection property]\label{Countable intersection property}
Let J be a countable index set, and for each $j \in J$ let $S_j$ be an $(\alpha_j , \beta, c, \rho)$-winning set, where $c>0$. Then, the set $S:= \bigcap_{j \in J} S_j$ is $(\alpha ,\beta, c, \rho)$-winning where $\alpha^c = \sum_{j \in J} \alpha_j^c$ (assuming that the series converges).
\end{proposition}

\begin{proposition}[Monotonicity]\label{Monotonicity}
If $S$ is $(\alpha , \beta, c, \rho)$-winning and $\tilde{\alpha} \geq \alpha$, $\tilde{\beta} \geq \beta$, $\tilde{c} \geq c$, and $\tilde{\rho} \geq \rho$, then $S$ is $(\tilde{\alpha} , \tilde{\beta}, \tilde{c}, \tilde{\rho})$-winning.
\end{proposition}

\begin{proposition}[Invariance under similarities]\label{Invariance under similarities}
Let $f:\R \to \R$ be a bijection that satisfies $$d(f(x),f(y))=\lambda d(x,y) \ \forall x,y \in \R.$$
Then a set $S$ is $(\alpha , \beta, c, \rho)$-winning if and only if the set $f(S)$ is $(\alpha , \beta, c, \lambda \rho)$-winning
\end{proposition}

\section{Proof of main theorems}

Broderick, Fishman and Simmons \cite{BFS} proved that, for the Cantor set $M_{\varepsilon}$ from Example \ref{excantor}, the ``filled in'' set $(-\infty,0) \cup M_{\varepsilon} \cup (1, +\infty)$ is $\left( \frac{2 \varepsilon}{(1- \varepsilon)\beta}, \beta, 0, \frac{\beta}{2} \right)$-winning for all $\beta \in (0,1)$. We generalize this to compact sets:

\begin{proposition}\label{Cantor winning}
Let $C$ be a compact set with $conv (C)=[0,1]$ and $\tau:= \tau(C)>0$. Then $S:=(-\infty,0) \cup C \cup (1, +\infty)$ is $\left( \frac{1}{\tau \beta}, \beta, 0, \frac{\beta}{2} \right)$-winning for all $\beta \in (0,1)$.
\end{proposition}

\begin{proof}
We have to describe Alice's strategy. Given a move for Bob $B$, how does Alice respond?
If there exists $n \in \N$ such that $B$ intersects $G_n$ and $|B| \leq \min \{ |L_n| ,|R_n| \}$, then $B \subset L_n \cup G_n \cup R_n$, so $B \cap G_n \neq \emptyset$ and $B \cap G_k = \emptyset$ for all $1 \leq k <n$. Alice erases $G_n$ if it is a legal movement. In any other case, Alice does not erase anything.

To show that this strategy is winning, suppose that $x_{\infty} \notin \bigcup_m A_m$. We want to see that $x_{\infty} \in S$.
By contradiction, suppose that $x_{\infty} \notin S$. Then there exists $n$ such that $x_{\infty} \in G_n$. We will show that Alice erases $G_n$ at some stage of the game.
By definition $x_{\infty} \in B_m$ for all $m \in \N_0$, and we assumed $x_{\infty} \in G_n$, so $x_{\infty} \in B_m \cap G_n$ for all $m \in \N_0$.
Since $\tau>0$, we have that $\min \{ |L_n|,|R_n| \}>0$. And also $\lim_{m \to \infty}|B_m|=0$, so taking $m_n \in \N_0$ to be the smallest integer such that $\min \{ |L_n|,|R_n| \} \geq |B_{m_n}|$, we know that $B_{m_n} \cap G_n \neq \emptyset$ and $B_{m_n} \cap G_k = \emptyset$ for all $1 \leq k <n$.
If $m_n=0$, then \[|B_0|=2\rho_0 \geq 2 \rho=\beta \geq \beta \min \{ |L_n|,|R_n| \}.\]
If $m_n>0$, then \[|B_{m_n}|\geq \beta |B_{m_n-1}|> \beta \min \{ |L_n|,|R_n| \}.\]
So, we have $|B_{m_n}|\geq \beta \min \{ |L_n|,|R_n| \}$. Hence, $$|G_n|\leq \frac{1}{\tau} \min \{ |L_n|,|R_n| \} \leq \frac{1}{\tau \beta} |B_{m_n}|=\alpha |B_{m_n}|.$$
This means that it is legal for Alice to erase $G_n$ in the $m_n$-th turn, and her strategy specifies that she does so.
If we suppose $m_i=m_j$ then the first complementary interval intersecting $B_{m_i}=B_{m_j}$ is $G_j$ and also $G_i$, so $i=j$. So, the elements of $\{m_n: \ n \in \N\}$ are all different.
\end{proof}

\begin{observation}
Let $C$ be a compact set with $conv (C)=[0,1]$ and $\tau:= \tau(C)>0$. Then, by the previous proposition and monotonicity, \[S:=(-\infty,0) \cup C \cup (1, +\infty)\] is a $\left( \frac{1}{\tau \beta}, \beta, c, \frac{\beta}{2} \right)$-winning set for all $\beta \in (0,1)$ and all $c \geq 0$.
\end{observation}

By following the proof given in \cite[Theorem 5.5]{BFS} closely, and making every step quantitative in our setting, we obtain the following result:

\begin{theorem}\label{teoexplicito}
Let $S\subset \R$ be an $(\alpha,\beta,c,\rho)$-winning set, with $0< c < 1$ and $0< \beta \leq \frac{1}{4}$. Then for all balls $B_0 \subset \R$ with radius larger or equal than $\rho$, we have:
\[\dim_{\rm H} (S \cap B_0 ) \geq 1-\frac{1440\alpha\log(6)}{|\log(\beta)|} > 0 \text{ if } \alpha^c \leq \frac{1}{720^2} (1 − \beta^{1−c}).\]
\end{theorem}

Since the proof of this theorem is very technical, we defer its proof in the Appendix A. Versions of Theorems \ref{finiteexplicit} and \ref{BilipExplicit} without an explicit constant can be deduced applying \cite[Theorem 5.5]{BFS} instead of Theorem \ref{teoexplicito}.


\subsection{Proof of Theorem \ref{finiteexplicit}}

\begin{proof}[Proof of Theorem \ref{finiteexplicit}]

Without loss of generality we can assume that $\text{conv}(C)=[0,1]$ and also that the finite set is $\{b_1, \cdots, b_n\} \subset [0, \frac{1}{8}]$, because we are trying to find a homothetic copy.

By Proposition \ref{Invariance under similarities} and \ref{Cantor winning}, we know that $S_i:=(-\infty, -b_i) \cup (C-b_i) \cup (1-b_i, +\infty)$ is  $(\frac{1}{\tau \beta}, \beta, c, \frac{\beta}{2})$-winning for all $\beta \in (0,1)$ and $c>0$. Let $\alpha:=\frac{1}{\tau \beta}$.

Then, by Proposition \ref{Countable intersection property}, $S:= \bigcap_{i=1}^n S_i$ is $(\frac{n^{\frac{1}{c}}}{\tau \beta}, \beta, c , \frac{\beta}{2})$-winning for all $\beta \in (0,1)$ and $c>0$. Taking $\beta:=\frac{1}{4}$, $c := 1- \frac{1}{\log(\alpha^{-1})} =1- \frac{1}{\log(\frac{\tau}{4})}$, $B:=[\frac{3}{8}, \frac{5}{8}]$ of radius $\frac{1}{8}=\frac{\beta}{2}=\rho$, we have $\alpha^c=e \alpha=\frac{e4}{\tau}$.

By using Theorem \ref{teoexplicito}, we have that:
\[
\text{dim}_{\rm H} (S \cap B)>0
\]
if \begin{equation}\label{cond cardinal} n \alpha^c \leq \frac{1}{720^2} (1 − \beta^{1−c}). \end{equation}

So, to guarantee the presence of a homothetic copy of a set of size $n$, it is sufficient that $n \leq \frac{\tau}{720^2 e 4}(1-\frac{1}{4}^{\frac{1}{\log(\frac{\tau}{4})}})$.
Since $f(\tau):=\log (\tau) \left( 1-4^{-\frac{1}{\log(\frac{\tau}{4})}} \right)$ is a decreasing function, and $\lim_{\tau \to +\infty}f(\tau)=\log(4)$, we get a simplified condition to guarantee the presence of a homothetic copy of a set of size $n$:
\[n \leq N(\tau):= \left\lfloor\frac{\log (4)}{4 e (720)^2}\frac{\tau}{\log (\tau)}\right\rfloor.\]

For those values of $n$, we have seen $\text{dim}_{\rm H}(S \cap B)>0$. For each $x\in S\cap B$, using $0 \leq b_i \leq \frac{1}{8}$, we have $$x+b_i \in (B+b_i)\cap(S+b_i) \subset \left[ \frac{3}{8}, \frac{6}{8} \right]\cap \left( (-\infty, 0) \cup C \cup (1, +\infty) \right).$$
Since $[\frac{3}{8}, \frac{6}{8}]$ is disjoint from $(-\infty, 0)$ and $(1, +\infty)$, we have that $x+b_i \in C$.

So $x+ \{ b_1, \cdots , b_n \}$ is a translated copy of the given finite set, which is contained in $C$.
\end{proof}

\subsection{Proof of Theorem \ref{BilipExplicit}}

In this section we prove Theorem \ref{BilipExplicit}. For this, we need to generalize the invariance under similarities (Proposition \ref{Invariance under similarities}) to bi-Lipschitz functions.

\begin{proposition}\label{bilip winning}
Let $f: \R \to \R$ be a bi-Lipschitz function with constants $(c_1, c_2)$.
If $S$ is a $(\alpha, \beta, 0, \rho)$-winning set, then $f(S)$ is a $(\frac{c_2}{c_1}\alpha, \frac{c_2}{c_1} \beta,0,c_2 \rho)$-winning set.
\end{proposition}

\begin{proof}
Bob plays a sequence $B_0 \supseteq B_1 \supseteq \cdots$ of nested intervals of length $|B_m|=2 \sigma_m$, with $\sigma_m \to 0$, $\sigma_{m+1}\geq \frac{c_2}{c_1} \beta \sigma_{m}$ and $\sigma_0 \geq \rho c_2$.

We want to define a strategy for Alice $(A_m)_m$ satisfying the rules of the $(\frac{c_2}{c_1}\alpha, \frac{c_2}{c_1}\beta,0,\rho)$-game and guaranteeing that if $x_{\infty}':= \bigcap_m B_m \notin \bigcup_m A_m$, then $x_{\infty}' \in f(S)$.

Since $(f^{-1}(B_m))_m$ is a sequence of nested closed intervals of length $| f^{-1}(B_m) |=:2 \rho_m \in \left[ \frac{2\sigma_m}{c_2}, \frac{2\sigma_m}{c_1} \right]$, then $x_{\infty}:=f^{-1}(x_{\infty}')=\bigcap_m f^{-1}(B_m)$.

We have that
\[\rho_{m+1} \geq \frac{\sigma_{m+1}}{c_2} \geq \frac{\beta \sigma_{m}}{c_1} \geq  \beta \rho_m \]
and $\rho_{0} \geq \frac{\sigma_{0}}{c_2} \geq  \rho$.
Since $S$ is a $(\alpha, \beta, 0, \rho)$-winning set, then Alice has a strategy $(A_m')_m$ satisfying the rules of the $(\alpha, \beta, 0, \rho)$-game and guaranteeing that if $x_{\infty}=\bigcap_m f^{-1}(B_m) \notin \bigcup_m A_m'$, then $x_{\infty} \in S$.

We have that $|A_m'|=:2\rho_m'$ where $\rho_m' \leq \alpha \rho_m$. We define $A_m:=f(A_m')$. Then, we have that $|A_m|=:2\sigma_m' \in \left[c_1 2\rho_m', c_2 2\rho_m' \right]$. So, \[\sigma_m' \leq c_2 \rho_m' \leq c_2 \alpha \rho_m \leq \frac{c_2}{c_1} \alpha \sigma_m.\]

If Alice's strategy is $(A_m)_m$, then as long as $x_{\infty}' \notin \bigcup_m A_m$, we have that $x_{\infty} \notin \bigcup_m A_m'$, and since $S$ is a $(\alpha, \beta, 0, \rho)$-winning set, we have that $x_{\infty} \in S$. So, $x_{\infty}'=f(x_{\infty}) \in f(S)$.
\end{proof}

Now, we are able to prove the theorem.

\begin{proof}[Proof of Theorem \ref{BilipExplicit}]
 Let $\beta:=\min\{ \frac{m}{D}; \frac{1}{4A}\}$, $\tilde{\beta}:=A\beta=\min\{ \frac{Am}{D}; \frac{1}{4}\} \in (0, \frac{1}{4}]$, $\alpha:=\frac{1}{\tau\beta}$ and $c := 1- \frac{1}{\log(\alpha^{-1})} =1- \frac{1}{\log(\tau \beta)}$ (so $\alpha^c=e\alpha=\frac{e}{\tau \beta}$).

By Proposition \ref{Cantor winning}, $S:=(-\infty,0) \cup C \cup (1, +\infty)$ is $\left( \alpha, \beta, 0, \frac{\beta}{2}\right)$-winning for all $\beta \in (0,1)$.
Then, by Proposition \ref{bilip winning}, for every $f$ in $\mathcal{F}$ we have that \[f^{-1}(S) \text{ is } \left( \frac{c_2(f)}{c_1(f)} \alpha, \frac{c_2(f)}{c_1(f)} \beta, 0, (c_1(f))^{-1} \frac{\beta}{2} \right)-\text{winning}.\]
So, by monotonicity (Proposition \ref{Monotonicity}), for every $f$ in $\mathcal{F}$ we have that \[f^{-1}(S) \text{ is } \left( A \alpha, A \beta, c, D\frac{\beta}{2} \right)-\text{winning } \forall c > 0, \ \forall \beta \in (0,\frac{1}{A}).\]
Then, by the countable intersection property (Proposition \ref{Countable intersection property}), \[\bigcap_{1 \leq i \leq n}f_i^{-1}(S)\] is $\left( n^{\frac{1}{c}}A \alpha, A \beta, c, D \frac{\beta}{2} \right)$-winning.

By hypothesis $\bigcap_{1 \leq i \leq n}f_i^{-1}([0,1])$ contains a closed interval $B$ of length $m$, and by definition of $\beta$ we know that $D \beta \in (0, m]$, so $\bigcap_{1 \leq i \leq n}f_i^{-1}([0,1])$ contains an interval of length $D \beta$.
Then, by applying Theorem \ref{teoexplicito}, we know that in order to guarantee the presence of any bi-Lipschitz pattern in the family $\mathcal{F}=\mathcal{F}_{A,D}$ of size $n$, it is a sufficient condition for $n$ to satisfy that \[n A^c \alpha^c \leq \frac{1}{720^2} (1 − \tilde{\beta}^{1−c}),\] so,
\[ n \leq  \frac{\tau \min\{ \frac{m}{D}; \frac{1}{4A}\}}{720^2 A^c e} \left( 1 − \min\{ \frac{Am}{D}; \frac{1}{4}\}^{\frac{1}{\log(\tau \beta)}} \right).\]
Note that when $\tau$ is large enough, we have $\frac{\tau \min\{ \frac{m}{D}; \frac{1}{4A}\}}{720^2 A^c e} \left( 1 − \min\{ \frac{Am}{D}; \frac{1}{4}\}^{\frac{1}{\log(\tau \beta)}} \right) \sim \frac{\tau}{\log(\tau)}$.

For those values of $n$, we know that $$\text{dim}_{\rm H} \left( \bigcap_{1 \leq i \leq n}f_i^{-1}(S) \cap B \right) > 0.$$ For each $x \in \bigcap_{1 \leq i \leq n}f_i^{-1}(S) \cap B$, we have $f_i(x) \in S$ for all $i$, and by hypothesis $x\in B \subset \bigcap_{1\leq i \leq n} f_i^{-1}([0,1])$, so $f_i(x) \in S \cap [0,1]=C$ for all $i$.
\end{proof}

\begin{observation}
The previous theorem implies Theorem \ref{finiteexplicit}: We can suppose that $conv(C)=[0,1]$ and the finite set $\{b_1, \cdots, b_n\}$ is contained in $[0, \frac{1}{8}]$. By taking $f_i(x):=x+b_i$, $c_1=c_2=A=D=1$, $\mathcal{F}:=\{x \mapsto x+b : \ b\in[0,\frac{1}{8}] \}$ and the closed interval $[\frac{3}{8}, \frac{5}{8}]$, the hypotheses of the previous theorem are satisfied.
\end{observation}


\appendix
\section{Proof of Theorem \ref{teoexplicito}}

We can assume without loss of generality that the radius of $B_0$ is $\rho$. We let $x_0$ be the center of $B_0$, \[\rho_n:=\beta^n \rho,\] \[E_n:=\frac{\rho_n}{2}\Z +x_0,\] \[\mathcal{E}_n:=\left\{ B(\frac{\rho_n}{2}z +x_0, \rho_n): \ z \in \Z \right\}.\]
We will take Bob's move of the $n$-turn from $\mathcal{E}_n$.

We also define \[D_n:=3 \rho_n \Z+x_0 \subset E_n,\] \[\mathcal{D}_n:=\{B(3 \rho_n z+x_0, \rho_n): \ z \in \Z \} \subset \mathcal{E}_n.\] Note that the elements of $\mathcal{D}_n$ are disjoint.

Let $N:=\lfloor \frac{1}{720\alpha}\rfloor$. We fix $\gamma \in (0,1)$, a small number to be determined later.

We define the function $\pi_n : \mathcal{E}_{n+1} \to \mathcal{E}_n$, $B \mapsto \pi_n(B)$ in the following way:
\begin{itemize}
\item When $n\neq jN$ for all $j$: we define $\pi_n(B)$ as the element of $\mathcal{E}_n$ that contains $B$ such that $B$ is as centered as possible inside that element.
\item When $n=jN$ for some $j$: If there exists $B'\in \mathcal{D}_{jN}$ containing $B$, we define $\pi_n(B):=B'$ (it is well defined because in that case there is only one element belonging to $\mathcal{D}_{jN}$). If not, we define the function as before.
\end{itemize}

Intuitively the function $\pi_n$ carries the elements of level $n+1$ to its ancestor of level $n$.

We use the following notation: for $m<n$ and $B \in \mathcal{E}_n$, $\pi_m(B):= \pi_m \circ \pi_{m+1} \circ \cdots \circ \pi_{n-1}(B) \in \mathcal{E}_m$. This is to say, we carry $B$ to its ancestor of level $m$ via the functions $\pi$.
If Bob plays $B \in \mathcal{E}_n$ in the turn $n$, we consider that in the previous turns $m \in \{0, \cdots, n-1\}$ Bob has played $\pi_m(B)$.
Then, we have the following inclusions of movements from the turn $n$ to the turn $0$: \[B \subset \pi_{n-1}(B) \subset \cdots \subset \pi_0(B).\]
Alice responds under her winning strategy. If in the turn $n$ Bob plays $B \in \mathcal{E}_n$, we define $\mathcal{A}(B)$ as Alice's answer (each $A \in \mathcal{A}(B)$ is a countable collection of sets $A:= \{A_{i,n}\}_i$, and a legal movement as an answer for $B$ i.e.: $\sum_i \rho_{i,n}^c \leq (\alpha \rho_n)^c$). Let \[\mathcal{A}^*_m(B):=\{A \in \mathcal{A}(\pi_m(B)) : B\cap A \neq \emptyset \}\] be Alice's answer (this is a list of sets) to the ancestor of $B$ of level $m<n$.

Given any ball $B$, we denote by $\frac{1}{2}B$ the ball with the same center as $B$ and the half of the radius.

Note that as $\beta \leq \frac{1}{4}$, if $B \in \mathcal{D}_{jN}$ and $B'\in \mathcal{E}_{jN+1}$ satisfying that $B' \cap \frac{1}{2}B \neq \emptyset$, then $B' \subset B$, so $\pi_{jN}(B')=B$. It follows that \begin{equation}\label{eqpi}\text{if } n>jN, \ B' \subset \frac{1}{2}B \text{ with } B'\in \mathcal{E}_n \text{ and } B \in \mathcal{D}_{jN}, \text{ then } \pi_{jN}(B')=B.\end{equation}
This is true because if we look at the ancestor of $B'$ of level $jN+1$, since $\pi_n$ chooses the element belonging to $\mathcal{E}_n$ that contains $B$ such that $B$ is as centered as possible, that element must intersect $\frac{1}{2}B$.

We define for every $B \in \mathcal{D}_j$ \[\phi_j (B):= \sum_{n<j} \ \sum_{A \in \mathcal{A}^*_n(B)} \rho_{i,n}^c.\] This is a measure of all of Alice's answers to the ancestors of $B$. Note that $\phi_0 (B)=0$.

Let
\[\mathcal{D}'_j:=\{ B \in \mathcal{D}_j: \ \phi_j(B) \leq (\gamma \rho_j)^c\}.\]
We define \[\mathcal{D}_j(B):=\{B' \in \mathcal{D}_j : \ B' \subset \frac{1}{2}B\}.\]

\subsubsection{Some useful bounds}

\begin{observation}\label{cota card 1}
If $B\in \mathcal{D}'_{jN}$ and if $N \geq 2$, by using that $\beta \leq \frac{1}{4}$, we have that
\begin{itemize}
\item $|\frac{1}{2}B|=\beta^{jN}\rho$,
\item $|B'|=2\beta^{(j+1)N}\rho$ for every $B' \in \mathcal{D}_{(j+1)N}$,
\item $|H|=\beta^{(j+1)N}\rho$ for every complementary interval $H$ between two consecutive intervals of $\mathcal{D}_{(j+1)N}$ .
\end{itemize}
Therefore,
\begin{align*}
\#\mathcal{D}_{(j+1)N}(B)&\geq \frac{|\frac{1}{2}B|}{|B'|+|H|}-2= \frac{\beta^{jN}\rho}{3\beta^{(j+1)N}\rho}-2\\
&=\frac{1}{3}\beta^{-N}-2=\beta^{-N}(\frac{1}{3}-2\beta^N) \geq \beta^{-N}\frac{7}{24}.
\end{align*}
\end{observation}

\begin{proposition}\label{cota card 2}
If $\alpha^c \leq \frac{1}{720^2}(1-\beta^{1-c})$, we have that \[\#(\mathcal{D}_{(j+1)N}(B)\setminus \mathcal{D}'_{(j+1)N}) \leq \frac{1}{12} \beta^{-N} \text{ for all } B\in \mathcal{D}'_{jN}.\]
\end{proposition}

We denote by $\text{rad}(B)$ the radius of the ball $B$. W start by proving two preliminary lemmas:
\begin{lemma}\label{cota card 3}
\begin{enumerate}[a)]
\item For all $n \in \N$ and $B' \in \mathcal{E}_n$ we have that
\[\sum_{A \in \mathcal{A}(B')} \min\left\{1, \frac{\rho_{i,n}^c}{(\gamma \rho_{(j+1)N})^c}\right\}\frac{\rho_{i,n}+2 \rho_{(j+1)N}}{\textup{rad}(B')}\leq 3 \alpha^c \max\left\{\alpha^{1-c}, \gamma^{-c} \left( \frac{\rho_{(j+1)N}}{\textup{rad}(B')} \right)^{1-c}\right\}.\]
\item If $B\in \mathcal{D}'_{jN}$ then \[\sum_{n<jN} \ \sum_{A \in \mathcal{A}^*_n(B)} \min\left\{1, \frac{\rho_{i,n}^c}{(\gamma \rho_{(j+1)N})^c}\right\}\frac{\rho_{i,n}+2 \rho_{(j+1)N}}{\textup{rad}(B)}\leq 3 \gamma^c \max\left\{\gamma^{1-c}, \gamma^{-c} \left( \frac{\rho_{(j+1)N}}{\textup{rad}(B)} \right)^{1-c}\right\}.\]
\end{enumerate}
\end{lemma}

\begin{proof}[Proof of Lemma \ref{cota card 3}]

Firstly, we will see that \begin{equation}\label{cotaxy}\min\left\{1,\frac{x^c}{(\gamma y)^c}\right\}(x+2y) \leq 3 x^c \max\left\{x^{1-c}, \frac{y^{1-c}}{\gamma^c}\right\} \text{ for all }x,y>0.\end{equation}
If $x \geq y$: \[\min\left\{1, \frac{x^c}{(\gamma y)^c}\right\}(x+2y)\leq x+2y \leq 3x\leq 3 x^c \max\left\{x^{1-c}, \frac{y^{1-c}}{\gamma^c}\right\}.\]
If $x < y$: \[\min\left\{1, \frac{x^c}{(\gamma y)^c}\right\}(x+2y)\leq \frac{x^c}{(\gamma y)^c} (x+2y)< \frac{x^c3y}{(\gamma y)^c}  \leq 3 x^c \max\left\{x^{1-c}, \frac{y^{1-c}}{\gamma^c}\right\}.\]

Secondly, we will prove that if $n \in \N$ and $B' \in \mathcal{E}_n$ then the claim a) holds:
By applying the inequality \eqref{cotaxy} to $x:=\frac{\rho_{i,n}}{\text{rad}(B')}$ and $y:=\frac{\rho_{(j+1)N}}{\text{rad}(B')}$, summing over all $A_{i,n} \in \mathcal{A}(B')$ and using that Alice is playing legally, we have that
\begin{align*}
&\sum_{A_{i,n} \in \mathcal{A}(B')} \min\left\{1, \frac{\rho_{i,n}^c}{(\gamma \rho_{(j+1)N})^c}\right\}\frac{\rho_{i,n}+2 \rho_{(j+1)N}}{\text{rad}(B')}\\
&\leq 3  \sum_{A_{i,n} \in \mathcal{A}(B')} \left(\frac{\rho_{i,n}}{\text{rad}(B')} \right)^c \max\left\{\left(\frac{\rho_{i,n}}{\text{rad}(B')}\right)^{1-c}, \gamma^{-c} \left(\frac{\rho_{(j+1)N}}{\text{rad}(B')}\right)^{1-c} \right\}\\
&\leq 3 \max\left\{ \left(\max_{A_{i,n} \in \mathcal{A}(B')} \frac{\rho_{i,n}}{\text{rad}(B')}\right)^{1-c}, \gamma^{-c} \left(\frac{\rho_{(j+1)N}}{\text{rad}(B')}\right)^{1-c} \right\} \sum_{A_{i,n} \in \mathcal{A}(B')} \left(\frac{\rho_{i,n}}{\text{rad}(B')} \right)^c\\
&\leq 3 \alpha^c \max\left\{\alpha^{1-c}, \gamma^{-c} \left(\frac{\rho_{(j+1)N}}{\text{rad}(B')}\right)^{1-c}\right\}.
\end{align*}

Finally, we prove the claim b). By applying the inequality \eqref{cotaxy} to $x:=\frac{\rho_{i,n}}{\text{rad}(B)}$ and $y:=\frac{\rho_{(j+1)N}}{\text{rad}(B)}$, summing over all element of $\bigcup_{n <jN}\mathcal{A}^*_n(B)$, and using that since $B\in \mathcal{D}'_{jN}$ we have \[\sum_{n<jN} \ \sum_{A_{i,n} \in \mathcal{A}^*_n(B)}\left(\frac{\rho_{i,n}}{\text{rad}(B)}\right)^c \leq \gamma^c,\] and in particular $\frac{\rho_{i,n}}{\text{rad}(B)} \leq \gamma$ for every $i$ and every $n<jN$. We obtain that:
\begin{align*}
&\sum_{n<jN} \ \sum_{A_{i,n} \in \mathcal{A}^*_n(B)} \min\left\{1, \frac{\rho_{i,n}^c}{(\gamma \rho_{(j+1)N})^c}\right\}\frac{\rho_{i,n}+2 \rho_{(j+1)N}}{\text{rad}(B)}\\
&\leq 3 \sum_{n<jN} \ \sum_{A_{i,n} \in \mathcal{A}^*_n(B)}\left(\frac{\rho_{i,n}}{\text{rad}(B)} \right)^c \max\left\{\left(\frac{\rho_{i,n}}{\text{rad}(B)} \right)^{1-c}, \gamma^{-c} \left(\frac{\rho_{(j+1)N}}{\text{rad}(B)}\right)^{1-c}\right\}\\
&\leq 3 \left( \sum_{n<jN} \ \sum_{A_{i,n} \in \mathcal{A}^*_n(B)}\left(\frac{\rho_{i,n}}{\text{rad}(B)} \right)^c\right) \max\left\{\gamma^{1-c}, \gamma^{-c} \left(\frac{\rho_{(j+1)N}}{\text{rad}(B)}\right)^{1-c}\right\}\\
&\leq 3 \gamma^c \max\left\{\gamma^{1-c}, \gamma^{-c} \left( \frac{\rho_{(j+1)N}}{\text{rad}(B)} \right)^{1-c}\right\}.
\end{align*}

\end{proof}

\begin{lemma}
If $B\in \mathcal{D}'_{jN}$, we have that for all $B'' \in \mathcal{D}_{(j+1)N}(B)$ and all $n \in \{jN, \cdots, (j+1)N-1 \}$, there exists $B' \in \mathcal{E}_n$ with $B' \subset B$ and $B''\subset \frac{1}{2}B'$, where $B'=B$ in the particular case $n=jN$.
\end{lemma}

\begin{proof}
\begin{itemize}
\item In the case $n=jN$: First, note that if $B''\in \mathcal{D}_{(j+1)N}(B)$, then $B'' \in \mathcal{E}_{(j+1)N}$, $B'' \subset \frac{1}{2}B$ where $B\in \mathcal{D}_{jN}$. If we take $B':=\pi_{jN}(B'') \in \mathcal{E}_{jN}$, we have from Equation \eqref{eqpi} that $B'=B$.

\item In the case $jN< n \leq (j+1)N$: We can write $n=jN+k$ with $k \in \{ 1, \cdots , N-1\}$.
Since $B \in \mathcal{D}_{jN}$, then \[B=B(3\rho_{jN}z+x_0, \rho_{jN}) \text{ for some } z \in \Z.\]
Since $B'' \in \mathcal{D}_{(j+1)N}(B)$, then \[B''=B(3\rho_{(j+1)N}z''+x_0, \rho_{(j+1)N}) \text{ for some } z'' \in \Z,\]
and also $B'' \subset \frac{1}{2}B$, where it is easy to see that the last inclusion is equivalent to $\frac{1}{2}-\beta^N \geq 3|\beta^N z''-z|$.

We want to find $B' \in \mathcal{E}_n$, i.e: $B'=B(\frac{\rho_n}{2}z'+x_0, \rho_n)$ with $z' \in \Z$ such that $B' \subset B$ and $B'' \subset \frac{1}{2}B'$.
The condition $B'' \subset \frac{1}{2}B'$ is equivalent to $\frac{1}{2}-\beta^{N-k} \geq |3z'' \beta^{N-k} - \frac{z'}{2}|$, and also to
\[ z' \in [ 6z'' \beta^{N-k} +1-2\beta^{N-k}, 6z'' \beta^{N-k} -1+2\beta^{N-k} ]=:I_1. \]
And the condition $B' \subset B$ is equivalent to $1-\beta^k \geq |-\frac{\beta^k}{2}z'+3z|$, and also to \[ z' \in [6 \beta^{-k}z-2(\beta^{-k}-1) , 6 z\beta^{-k}+2(\beta^{-k}-1) ]=:I_2.\]

Therefore, what we must prove is that given $z, z'' \in \Z$ satisfying $\frac{1}{2}-\beta^N \geq 3|\beta^N z''-z|$, there exists $z' \in \Z \cap I_1 \cap I_2$. For this it is  sufficient to see that the length of $I_1 \cap I_2$ is larger or equal to $1$.

We take $c_1:= 6z'' \beta^{N-k}$ and $c_2:= 6 z\beta^{-k}$ the centers of the intervals $I_1$ and $I_2$ respectively,
and $r_1:=1-2\beta^{N-k}$ and $r_2:=2(\beta^{-k}-1)$ their respective radii.

The length of $I_1$ is $2r_1=2(-1+2\beta^{N-k})\geq 2(1-\frac{2}{4})=1$ because $\beta \in (0, \frac{1}{4}]$ and $N-k \geq 1$, and the length of $I_2$ is $2r_2=4(\beta^{-k}-1)\geq 12$ because $\beta \in (0, \frac{1}{4}]$ and $k \geq 1$.

Since the length of $I_1$ is larger or equal than $1$, it is sufficient to see that $\text{dist}(c_1, c_2)\leq r_2 - r_1$.
Since
\begin{align*}
\text{dist}(c_1 , c_2)&=|6 z\beta^{-k}- 6z'' \beta^{N-k}|=6 \beta^{-k}|z-z''\beta^N|\\
&\leq 2 \beta^{-k} (\frac{1}{2}-\beta^N)= \beta^{-k}-2\beta^{N-k}.
\end{align*}
So,
\begin{align*}
r_2-r_1&=2\beta^{-k}-2-1+2\beta^{N-k}\\
&\geq \text{dist}(c_1, c_2)+(4 \beta^{N-k}+\beta^{-k}-3)\geq \text{dist}(c_1, c_2),
\end{align*}
because $4 \beta^{N-k}+\beta^{-k}-3\geq 0$ which follows from $1 \geq 3 \beta^k$ ($\beta \in (0, \frac{1}{4}]$ and $k \geq 1$).

\end{itemize}
\end{proof}

Now we are ready to prove Proposition \ref{cota card 2}.

\begin{proof}[Proof of Proposition \ref{cota card 2}]

\begin{align}\label{cotasetminus1}
\#(\mathcal{D}_{(j+1)N}(B)\setminus \mathcal{D}'_{(j+1)N})&\leq\# \left\{ B' \in \mathcal{D}_{(j+1)N}(B) : \ \frac{\phi_{(j+1)N}(B')}{(\gamma \rho_{(j+1)N})^c} >1 \right\}\nonumber\\
&\leq \sum_{B' \in \mathcal{D}_{(j+1)N}(B)} \min\left\{1, \frac{\phi_{(j+1)N}(B')}{(\gamma \rho_{(j+1)N})^c}\right\}\nonumber\\
&\leq \sum_{B' \in \mathcal{D}_{(j+1)N}(B)} \min\left\{1, \sum_{n<(j+1)N}\sum_{A_{i,n} \in \mathcal{A}^*_n(B')} \frac{\rho_{i,n}^c}{(\gamma \rho_{(j+1)N})^c}\right\}\nonumber\\
&\leq \sum_{B' \in \mathcal{D}_{(j+1)N}(B)} \sum_{n<(j+1)N} \sum_{A_{i,n} \in \mathcal{A}^*_n(B')}  \min\left\{1,\frac{\rho_{i,n}^c}{(\gamma \rho_{(j+1)N})^c}\right\} \nonumber\\
&\leq \sum_{B' \in \mathcal{D}_{(j+1)N}(B)} \sum_{n<jN} \sum_{A_{i,n} \in \mathcal{A}^*_n(B')}  \min\left\{1,\frac{\rho_{i,n}^c}{(\gamma \rho_{(j+1)N})^c}\right\}\nonumber\\
& \hspace{0.5cm} +\sum_{B' \in \mathcal{D}_{(j+1)N}(B)} \sum_{jN \leq n<(j+1)N} \ \sum_{A_{i,n} \in \mathcal{A}^*_n(B')}  \min\left\{1,\frac{\rho_{i,n}^c}{(\gamma \rho_{(j+1)N})^c}\right\}.
\end{align}

We split the last sum into two sums, depending on whether $n<jN$ or $jN \leq n <(j+1)N$.

To get a bound for the left-hand side sum of \eqref{cotasetminus1} we will use that if $n<jN$ then
\begin{align*}
&\left\{(B', A):\ B' \in \mathcal{D}_{(j+1)N}(B), A\in \mathcal{A}^*_n(B') \right\} \\
&\subset \left\{(B', A):\ B' \in \mathcal{D}_{(j+1)N}(B), A\in \mathcal{A}^*_n(B), A\cap B' \neq \emptyset \right\}.
\end{align*}
Since $B\in \mathcal{D}'_{jN}$ the set $\mathcal{A}^*_{n}(B)$ only makes sense for $n<jN$.
This inclusion holds because of Equation \eqref{eqpi}, because $\mathcal{A}(\pi_n(B')) \subset \mathcal{A}(\pi_n(B))$ since $B' \subset B$.

So, we have
\begin{align*}
&\sum_{B' \in \mathcal{D}_{(j+1)N}(B)} \sum_{A_{i,n} \in \mathcal{A}^*_n(B')} \min\left\{1,\frac{\rho_{i,n}^c}{(\gamma \rho_{(j+1)N})^c}\right\}\\
&\leq  \sum_{A_{i,n} \in \mathcal{A}^*_n(B)} \sum_{B' \in \mathcal{D}_{(j+1)N}(B) \atop B'\cap A\neq \emptyset} \min\left\{1,\frac{\rho_{i,n}^c}{(\gamma \rho_{(j+1)N})^c}\right\}\\
&= \sum_{A_{i,n} \in \mathcal{A}^*_n(B)} \min\left\{1,\frac{\rho_{i,n}^c}{(\gamma \rho_{(j+1)N})^c}\right\} \#\{B' \in \mathcal{D}_{(j+1)N}(B), B'\cap A\neq \emptyset\}.
\end{align*}

Now, we will get a bound for the right-hand side of Equation \eqref{cotasetminus1}, this is when $jN \leq n < (j+1)N$. Recall that $B \in \mathcal{D}'_{jN}$.
First, note that if $B''\in \mathcal{D}_{(j+1)N}(B)$, then $B'' \in \mathcal{E}_{(j+1)N}$, $B'' \subset \frac{1}{2}B$ where $B\in \mathcal{D}_{jN}$. If we take $B':=\pi_{jN}(B'') \in \mathcal{E}_{jN}$, by Equation \eqref{eqpi} we have that $B'=B$.

Since for all $B'' \in \mathcal{D}_{(j+1)N}(B)$, there exists $B' \in \mathcal{E}_n$ with $B' \subset B$ and $B''\subset \frac{1}{2}B'$ ($B'=B$ if $n=jN$). Hence
\begin{align*}
&\sum_{B'' \in \mathcal{D}_{(j+1)N}(B)}  \sum_{A_{i,n} \in \mathcal{A}^*_n(B'')} \min\left\{1,\frac{\rho_{i,n}^c}{(\gamma \rho_{(j+1)N})^c}\right\}\\
&\leq \sum_{B' \in \mathcal{E}_n \atop B'\subset B}   \sum_{B'' \in \mathcal{D}_{(j+1)N}(B')}   \sum_{A \in \mathcal{A}(B') \atop A\cap B''\neq \emptyset} \min\left\{1,\frac{\rho_{i,n}^c}{(\gamma \rho_{(j+1)N})^c}\right\}\\
&= \sum_{B' \in \mathcal{E}_n \atop B'\subset B}  \sum_{A \in \mathcal{A}(B')}  \sum_{B'' \in \mathcal{D}_{(j+1)N}(B') \atop A\cap B''\neq \emptyset} \min\left\{1,\frac{\rho_{i,n}^c}{(\gamma \rho_{(j+1)N})^c}\right\}\\
&=\sum_{B' \in \mathcal{E}_n \atop B'\subset B}  \sum_{A \in \mathcal{A}(B')} \min\left\{1,\frac{\rho_{i,n}^c}{(\gamma \rho_{(j+1)N})^c}\right\} \#\{B'' \in \mathcal{D}_{(j+1)N}(B'): \ A\cap B''\neq \emptyset\},
\end{align*}
where the inequality holds by considering in particular $B':=\pi_n(B'')\subset B$.

Then, by inequality \eqref{cotasetminus1}, and what we have seen before, we have

\begin{align}\label{cotasetminus2}
&\#(\mathcal{D}_{(j+1)N}(B)\setminus \mathcal{D}'_{(j+1)N})\nonumber\\
&\leq \sum_{B' \in \mathcal{D}_{(j+1)N}(B)}  \sum_{n<jN}  \sum_{A_{i,n} \in \mathcal{A}^*_n(B')}  \min\left\{1,\frac{\rho_{i,n}^c}{(\gamma \rho_{(j+1)N})^c}\right\}+\nonumber\\
&\sum_{B' \in \mathcal{D}_{(j+1)N}(B)}  \sum_{n=jN}^{(j+1)N-1}  \sum_{A_{i,n} \in \mathcal{A}^*_n(B')}  \min\left\{1,\frac{\rho_{i,n}^c}{(\gamma \rho_{(j+1)N})^c}\right\}\nonumber\\
&\leq  \sum_{n<jN}  \sum_{A_{i,n} \in \mathcal{A}^*_n(B)}  \min\left\{1,\frac{\rho_{i,n}^c}{(\gamma \rho_{(j+1)N})^c}\right\} \#\{B' \in \mathcal{D}_{(j+1)N}(B): \ B'\cap A \neq \emptyset\} +\nonumber\\
&\sum_{n=jN}^{(j+1)N-1}  \sum_{B' \in \mathcal{E}_n \atop B'\subset B}  \sum_{A_{i,n} \in \mathcal{A}^*_n(B')}  \min\left\{1,\frac{\rho_{i,n}^c}{(\gamma \rho_{(j+1)N})^c}\right\} \#\{B'' \in \mathcal{D}_{(j+1)N}(B'): \ B''\cap A \neq \emptyset\}\nonumber\\
&\leq  \frac{\text{rad}(B)}{\rho_{(j+1)N}} \sum_{n<jN}\sum_{A_{i,n} \in \mathcal{A}^*_n(B)}  \min\left\{1,\frac{\rho_{i,n}^c}{(\gamma \rho_{(j+1)N})^c}\right\} \frac{\rho_{i,n}+2\rho_{(j+1)N}}{\text{rad}(B)} +\nonumber\\
&\sum_{n=jN}^{(j+1)N-1}  \sum_{B' \in \mathcal{E}_n \atop B'\subset B} \frac{\text{rad}(B')}{\rho_{(j+1)N}}  \sum_{A_{i,n} \in \mathcal{A}^*_n(B')}  \min\left\{1,\frac{\rho_{i,n}^c}{(\gamma \rho_{(j+1)N})^c}\right\} \frac{\rho_{i,n}+2\rho_{(j+1)N}}{\text{rad}(B')},
\end{align}
where in the first term of the last inequality we use that $B \in \mathcal{D}'_{jN}$ ($\phi_{jN}(B)\leq (\gamma \rho_{jN})^c$), in the second term that Alice is playing legally ($\sum_i \rho_{i,m}^c \leq (\alpha \rho_m^c)$), and in both terms that:
for every $B'\in \bigcup_n \mathcal{E}_n$ and every $A_{i,n}$, since the elements of $\mathcal{D}_{(j+1)N}$ are disjoint, $\mathcal{L}(B'')=2 \rho_{(j+1)N}$, and if moreover $B'' \cap A_{i,n} \neq \emptyset$ we have $B'' \subset \mathcal{N}(A_{i,n}, 2\rho_{(j+1)N})$. Therefore,
\begin{align*}
&\#\{B''\in \mathcal{D}_{(j+1)N}(B'): \ B''\cap A_{i,n} \neq \emptyset\}2 \rho_{(j+1)N}
=\mathcal{L}\left( \bigcup_{B''\in \mathcal{D}_{(j+1)N}(B') \atop B''\cap A_{i,n} \neq \emptyset} B''\right)\\
&\leq \mathcal{L}\left(\mathcal{N}(A_{i,n}, 2\rho_{(j+1)N}) \right)
= \mathcal{L}(A_{i,n})+4 \rho_{(j+1)N}=2\rho_{i,n}+4 \rho_{(j+1)N},
\end{align*}
which is saying \[\#\{B''\in \mathcal{D}_{(j+1)N}(B'): \ B''\cap A_{i,n} \neq \emptyset\} \leq \frac{1}{\rho_{(j+1)N}}(\rho_{i,n}+2 \rho_{(j+1)N}).\]

By inequality \eqref{cotasetminus2},  using claim b) from Lemma \ref{cota card 3} to bound the first term and claim a) from Lemma \ref{cota card 3} to bound the second one, we have that:
\begin{align}\label{cotasetminus3}
&\#(\mathcal{D}_{(j+1)N}(B)\setminus \mathcal{D}'_{(j+1)N})\nonumber\\
&\leq  \frac{\text{rad}(B)}{\rho_{(j+1)N}} \sum_{n<jN}  \sum_{A_{i,n} \in \mathcal{A}^*_n(B)}  \min\left\{1,\frac{\rho_{i,n}^c}{(\gamma \rho_{(j+1)N})^c}\right\} \frac{\rho_{i,n}+2\rho_{(j+1)N}}{\text{rad}(B)} \nonumber\\
&+ \sum_{jN \leq n<(j+1)N}  \sum_{B' \in \mathcal{E}_n \atop B'\subset B} \frac{\text{rad}(B')}{\rho_{(j+1)N}} \ \sum_{A_{i,n} \in \mathcal{A}^*_n(B')}  \min\left\{1,\frac{\rho_{i,n}^c}{(\gamma \rho_{(j+1)N})^c}\right\} \frac{\rho_{i,n}+2\rho_{(j+1)N}}{\text{rad}(B')}\nonumber\\
&\leq \frac{\text{rad}(B)}{\rho_{(j+1)N}} 3 \gamma^c \max\{\gamma^{1-c}, \gamma^{-c} \left( \frac{\rho_{(j+1)N}}{\text{rad}(B)} \right)^{1-c}\}\nonumber\\
& + \sum_{jN \leq n<(j+1)N}  \sum_{B' \in \mathcal{E}_n \atop B'\subset B} \frac{\text{rad}(B')}{\rho_{(j+1)N}} 3 \alpha^c \max\{\alpha^{1-c}, \gamma^{-c} (\frac{\rho_{(j+1)N}}{\text{rad}(B')})^{1-c}\}
\end{align}

To continue estimating, we will use that \[\frac{\text{rad}(B)}{\rho_{(j+1)N}}=\frac{\beta^{jN}\rho}{\beta^{(j+1)N}\rho}=\beta^{-N}.\]
To bound the second inequality we write $n=(j+1)N-k$ for some $k \in \{1, \cdots, N\}$. We know that $B:=B(3\rho_{jN}z+x_0, \rho_{jN})$ for some $z \in \Z$, and recall that $\mathcal{E}_n:=\{B(\frac{\rho_n}{2}z'+x_0, \rho_n): \ z' \in \Z \}$.
So,
\[\frac{\text{rad}(B')}{\rho_{(j+1)N}}=\beta^{-k}\] and
\begin{align*}
\#\{B' \in \mathcal{E}_n : \ B'\subset B \}&\leq \#\{\frac{\rho_n}{2}z'+x_0 \in B: \ z' \in \Z \}\\
&=\#\{z' \in \Z: \ 3\rho_{jN}z+x_0-\rho_{jN} \leq \frac{\rho_n}{2}z'+x_0 \leq 3\rho_{jN}z+x_0+\rho_{jN}\}\\
&=\#\{z' \in \Z: \ \frac{2}{\rho_n}\rho_{jN}(3z-1) \leq z' \leq \frac{2}{\rho_n}\rho_{jN}(3z+1)\}\\
&\leq \frac{2}{\rho_n}\rho_{jN}(3z+1) - \frac{2}{\rho_n}\rho_{jN}(3z-1) +1\\
&=\frac{4}{\beta^{N-k}}+1 =\frac{4}{\beta^{N-k}}+\frac{\beta^{N-k}}{\beta^{N-k}}\leq \frac{5}{\beta^{N-k}}
\end{align*}

Putting all together,
\begin{align}\label{cotasetminus4}
&\#(\mathcal{D}_{(j+1)N}(B)\setminus \mathcal{D}'_{(j+1)N})\nonumber\\
&\leq \frac{\text{rad}(B)}{\rho_{(j+1)N}} 3 \gamma^c \max\{\gamma^{1-c}, \gamma^{-c} \left( \frac{\rho_{(j+1)N}}{\text{rad}(B)} \right)^{1-c}\}\nonumber\\
& \hspace{0.5cm} + \sum_{jN \leq n<(j+1)N} \ \sum_{B' \in \mathcal{E}_n \ B'\subset B} \frac{\text{rad}(B')}{\rho_{(j+1)N}} 3 \alpha^c \max\{\alpha^{1-c}, \gamma^{-c} (\frac{\rho_{(j+1)N}}{\text{rad}(B')})^{1-c}\}\nonumber\\
&\leq \beta^{-N} 3 \gamma^c \max\{\gamma^{1-c}, \gamma^{-c} \beta^{N(1-c)}\} + \sum_{1 \leq k \leq N} 5 \beta^{-(N-k)} \beta^k 3 \alpha^c \max\{\alpha^{1-c}, \gamma^{-c} \beta^{k(1-c)}\}\nonumber\\
&\leq \beta^{-N} 3 \left( \max\{\gamma, \beta^{N(1-c)}\} + 5 \left( N \alpha + \alpha^c \gamma^{-c} \sum_{1 \leq k \leq N} \beta^{k(1-c)} \right) \right),\nonumber\\
\end{align}
where in the last inequality we have used that if $a_n,b_n \geq 0$ then $\sum_n \max\{a_n, b_n\} \leq \sum_n a_n + \sum_n b_n$.

We take $\gamma :=\frac{1}{72}$. If we saw that
\begin{enumerate}
\item $5 N \alpha \leq \frac{1}{144}$,
\item $5 \gamma^{-c} \sum_{k \in N_0} \beta^{k(1-c)}\leq \frac{1}{144(72)^c}=\frac{1}{10 (72)^{1+c}}$,
\item $\beta^{N(1-c)} \leq \frac{1}{72}$,
\end{enumerate}
we would have that \[\#(\mathcal{D}_{(j+1)N}(B)\setminus \mathcal{D}'_{(j+1)N}) \leq \beta^{-N}\frac{1}{12}.\]

Let's prove (1)--(3):
\begin{enumerate}
\item Since $N:=\lfloor \frac{1}{720} \alpha^{-1} \rfloor$, then $5 N \alpha \leq \frac{1}{144}$.
\item 
By hypothesis and by using that $c \in (0,1)$, $\beta \in (0,\frac{1}{4}]$, 
we have
\[\alpha^c \frac{1}{1-\beta^{1-c}} \leq (\frac{1}{720})^2 \leq (\frac{1}{720})^{1+c}\leq \frac{1}{10(72)^{1+c}},\]
so the second claim holds.
Moreover, since we have $\alpha^c \leq \alpha^c \frac{1}{1-\beta^{1-c}}\leq (\frac{1}{720})^{1+c}$, then  $2^c \leq 720 \leq (\frac{1}{720\alpha})^c$, so $N \in \N_{\geq 2}$.
Therefore \begin{equation}\label{eq2estrellas} N:=\lfloor \frac{1}{720\alpha} \rfloor \geq  \frac{1}{1440 \alpha}\end{equation}

On the other hand, we have by using the hypothesis and $\alpha, c \in (0,1)$, that \[\alpha \leq \alpha^c \leq \frac{1}{720^2} (1- \beta^{1-c})\leq \frac{1}{720^2}|\log(\beta^{1-c})|= \frac{1}{720^2}(1-c) |\log(\beta)|,\]
where in the last inequality we have used that $1-c \in (0,1)$, $\beta \in (0,\frac{1}{4}]$, $z:=\beta^{1-c} \in (0,1)$, and $f(z):=\log(\frac{1}{z})+z+1$ is a positive function on $(0,1)$, so $1-z \leq \log(\frac{1}{z})$.
Then, \begin{equation}\label{eq1estrella}N\alpha 720^2 \leq N (1-c) |\log(\beta)|.\end{equation}

\item For the third claim it is equivalent to prove that $N(1-c)|\log(\beta)|\geq |\log(\frac{1}{72})|$. Since $720^2\geq 1440 \log(72)$, by inequalities \eqref{eq2estrellas} and \eqref{eq1estrella} we have
\[N(1-c)|\log(\beta)|\geq N \alpha 720^2 \geq \frac{1}{1440} 720^2 \geq |\log(\frac{1}{72})|.\]
\end{enumerate}

\end{proof}

\begin{lemma}\label{canthijos}
If $\alpha^c \leq \frac{1}{720^2}(1-\beta^{1-c})$, we have \[\#(\mathcal{D}_{(j+1)N}(B)\cap \mathcal{D}'_{(j+1)N}) \geq \frac{1}{6} \beta^{-N} \text{ for every } B\in \mathcal{D}'_{jN}.\]
\end{lemma}
\begin{proof}
By Observation \ref{cota card 1} and Proposition \ref{cota card 2}, we have
\begin{align*}
\#(\mathcal{D}_{(j+1)N}(B)\cap \mathcal{D}'_{(j+1)N})&= \#\mathcal{D}_{(j+1)N}(B) - \#(\mathcal{D}_{(j+1)N}(B)\setminus \mathcal{D}'_{(j+1)N})\\
&\geq \frac{1}{4}\beta^{-N}-\frac{1}{12}\beta^{-N}=\frac{1}{6}\beta^{-N}.
\end{align*}
\end{proof}

\subsubsection{Conclusion of the proof}

By definition, $B_0 \in \mathcal{D}_0$. Moreover, $\phi_{0} (B_0):=0<(\gamma \rho)^c$, so $B_{0} \in \mathcal{D}'_{0}$.

We will construct a Cantor set $F$ in the following way:
\begin{itemize}
\item $\mathcal{B}_{0}:=\{B_{0}\}\subset \mathcal{D}'_{0}$.
\item Given a collection $\mathcal{B}_{j} \subset \mathcal{D}'_{jN}$ we construct a new collection of sets $\mathcal{B}_{j+1} \subset \mathcal{D}'_{(j+1)N}$ by replacing each element of $B \in \mathcal{B}_{j}$ by $M:=\lceil \frac{1}{6}\beta^{-N} \rceil$ elements of $\mathcal{D}_{(j+1)N}(B) \cap \mathcal{D}'_{(j+1)N}$ (this is possible by Lemma \ref{canthijos}).
\end{itemize}
We define \[F:=\bigcap_{j \in \N_0} \bigcup_{B \in \mathcal{B}_j}B.\]
By a standard argument (see e.g. \cite[Example 4.6]{FalconerBook}),
\begin{align*}
\dim_{\rm H} (F)&\geq \frac{\log(M)}{|\log(\beta^{N})|} \geq \frac{\log(\frac{1}{6}\beta^{-N})}{N |\log(\beta)|}\\
&=1-\frac{\log(6)}{N |\log(\beta)|}\geq 1-\frac{1440\alpha\log(6)}{|\log(\beta)|},
\end{align*}
where we used $N\geq \frac{1}{1440\alpha}$ from Equation \eqref{eq2estrellas}.

What remains to be proved is that $F \subset S \cap B_0$, because we would then get \[\dim_{\rm H} (S) \geq \dim_{\rm H} (S\cap B_{0}) \geq \dim_{\rm H} (F) \geq 1-\frac{1440\alpha\log(6)}{|\log(\beta)|}.\]

Clearly $F \subset B_0$, by definition of $F$. We have to see that $F \subset S$. Let $x \in F$. For every $j \in \N$ there exists a unique $B_{jN}\in \mathcal{B}_j$ containing $x$.
By definition of $\mathcal{B}_{j+1}$ we have that $B_{(j+1)}\subset \frac{1}{2}B_{jN}$.  By Equation \eqref{eqpi} we have $\pi_{jN}(B_{(j+1)N})=B_{jN}$.
The sequence $(B_{jN})_{j}$ can be extended in a unique way to a sequence $(B_n)_{n}$ satisfying $B_n \in \mathcal{E}_n$ and $B_n:=\pi_n(B_{n+1})$ for all $n$. We interpret this sequence as Bob's moves, to which Alice responds according to her winning strategy.

Thus, for each $x \in F$ we construct a sequence $(B_n)_n$ as before, where $x$ is the only element belonging to $\bigcap_n B_n$ (so, $x=x_{\infty}$ is the outcome of the game). We will see that $x \in S$ by contradiction. We suppose that $x \notin S$ where $S$ is an $(\alpha, \beta, c, \rho)$-winning set.
Then, $x \in \bigcup_{m \in \N_0} \bigcup_i A(\rho_{i,m}, h_{i,m})$, where $\sum_i \rho_{i,m}^c \leq (\alpha \rho_m)^c=(\alpha \beta^m \rho)^c$ (since it is a legal move for Alice we know that $\bigcup_i A(\rho_{i,m}, h_{i,m}) \in \mathcal{A}(B_m)$).
So $x \in A \in \mathcal{A}(B_m)$ for some $m$. And since $x \in B_m$, we have $x \in A \cap B_m$.
Since for every $n>m$ it holds that $\mathcal{A}^*_m(B_n)=\mathcal{A}(B_m)$ (because $\pi_m(B_n)=B_m$), then $\phi_j (B_{jN})\geq \rho_A^c$ for every $j$ such that $jN>m$ (because $\rho_A^c$ is just one term in the sum of the definition of $\phi_j (B_{jN})$ when $A \in \mathcal{A}^*_m(B_n)$).

On the other hand, since $B_{jN} \in \mathcal{D}'_j$, then $\phi_j(B_{jN})\leq (\gamma \rho_{jN})^c$. Then, by putting everything together, we have that $\rho_A \leq \gamma \rho_{jN}$ for all $j$ such that $jN>m$.
Letting $j \to \infty$, we get $\rho_A=0$ which is a contradiction. So, $x \in S$ and we have seen $F \subset S$.

We will prove that \[1-\frac{1440\alpha\log(6)}{|\log(\beta)|}>0 \text{ if } \alpha^c \leq \frac{1}{720^2} (1-\beta^{1-c}).\]

By using the hypotheses: $\alpha^c \leq \frac{1}{720^2} (1 − \beta^{1−c})$, $c \in (0,1)$, $\beta \in (0, \frac{1}{4}]$, the fact that $\alpha<1$ and that $f(z):=z-1-\log(z)$ is a positive function on $(0,1)$ (applied to $z:=\beta^{1-c} \in (0,1)$), we have that:

\begin{align*}
\frac{\alpha}{|\log(\beta)|}&\leq \frac{\alpha^c}{|\log(\beta)|}\leq \frac{1}{720^2} \frac{(1-\beta^{1-c})}{|\log(\beta)|}\\
&\leq \frac{1}{720^2} \frac{|\log(\beta^{1-c})|}{|\log(\beta)|}=\frac{1-c}{720^2}\leq \frac{1}{720^2}<\frac{1}{1440 \log(6)}.
\end{align*}
Then, $1440 \log(6) \frac{\alpha}{|\log(\beta)|}<1$. 

With this we conclude the proof of  Theorem \ref{teoexplicito}.





\end{document}